\documentclass[12pt,twoside,reqno]{amsart}
\usepackage{amsfonts}
\usepackage{amssymb}
\usepackage[notcite,notref]{
}

\makeatletter
\@namedef{subjclassname@2010}{%
  \textup{2010} Mathematics Subject Classification}
\makeatother

\numberwithin{equation}{section}

\newtheorem{teo}{Theorem}[section]
\newtheorem{prop}[teo]{Proposition}
\newtheorem{lem}[teo]{Lemma}
\newtheorem{cor}[teo]{Corollary}

\theoremstyle{definition}

\newtheorem{rem}[teo]{Remark}

\numberwithin{equation}{section}

\def\a{\alpha}
\def\b{\beta}

\def\o{\omega}
\def\R{\mathbb{R}}
\def\N{\mathbb{N}}
\def\Z{\mathbb{Z}}
\def\d{\delta}
\def\e{\varepsilon}
\def\t{\theta}
\def\f{\varphi}
\def\s{\sigma}
\def\mes{\operatorname{mes}}
\def\D{\Delta}

\begin{document}

\title[EMBEDDING THEOREMS]{Embedding theorems for   Sobolev and Hardy-Sobolev spaces and estimates of Fourier transforms}

\author[V.I. Kolyada]{V.I. Kolyada}
\address{Department of Mathematics\\
Karlstad University\\
Universitetsgatan 1 \\
651 88 Karlstad\\
SWEDEN} \email{viktor.kolyada@gmail.com}



\begin{abstract}

We prove embeddings of Sobolev and Hardy-Sobolev spaces into Besov spaces built upon certain mixed norms.
This gives an improvement of the known embeddings into usual Besov spaces. Applying these results,
we obtain Oberlin type estimates of Fourier transforms for functions in Sobolev spaces $W_1^1(\R^n).$

\end{abstract}

\maketitle

\section{Introduction}

This paper is devoted to the study of some inequalities for functions in the Sobolev spaces $W_p^1(\R^n)$ and Hardy-Sobolev spaces $HW_1^1(\R^n)$.

The Sobolev space  $W_p^1(\R^n)$ $(1\le p<\infty)$ is defined as the class  of all functions $f\in L^p(\R^n)$ for which
every first-order  weak derivative exists and belongs to $L^p(\R^n).$
The  classical Sobolev theorem (see \cite[Ch. V]{St})  states the
following.
\begin{teo}\label{Sob1} Let $n\ge 2,$
$1\le p< n,$ and $p^*=np/(n-p).$ Then for any $f\in
W_p^1(\mathbb{R}^n)$
\begin{equation}\label{sob1}
||f||_{p^*}\le c \|\nabla f\|_p.
\end{equation}
\end{teo}

The Lebesgue norm at the left-hand side of (\ref{sob1}) can be replaced by the stronger Lorentz norm.
Namely,  for any $f\in W_p^1(\R^n),\,\, n\ge 2,$ $1\le p< n,$
\begin{equation}\label{embed0}
||f||_{p^*,p}\le c ||\nabla f||_p
\end{equation}
(see  \cite{Al}, \cite{Neil}, \cite{Pe}, \cite{Po}).

Let a function $f$ be defined on $\R^n$ and let
$k\in\{1,...,n\}.$ Set
\begin{equation}\label{Delta}
\Delta_k(h)f(x)= f(x+he_k)-f(x), \quad x\in\R^n,\, h\in \R
\end{equation}
($e_k$ is the $k$th unit coordinate vector).

The following theorem holds.
\begin{teo}\label{1988} Let $n\in \N.$ Assume that $1<p<\infty$ and $n\ge 1$, or $p=1$ and $n\ge 2.$
If $p<q<\infty$ and $s=1-n(1/p-1/q)>0,$ then
 for any $f\in W_p^1(\R^n)$
\begin{equation}\label{embed1}
\sum_{k=1}^n \left(\int_0^\infty h^{-sp}||\D_k(h)f||^p_{q,p}\frac{dh}{h}\right)^{1/p}\le c \sum_{k=1}^n||D_k f||_p.
\end{equation}
\end{teo}
For $p>1$ inequality (\ref{embed1}) (with the weaker norm $||\D_k(h)f||_q$ at the left-hand side) was obtained
by Herz \cite{Herz}. For $p=1, n\ge 2$ Theorem \ref{1988} was proved in \cite{K88} (see also \cite{K93}). The case $p=1$ is of special interest; we stress that
 Theorem \ref{1988} fails for $p=n=1.$ However, this theorem holds for any function $f$ from the Hardy space $H^1(\R)$
such that $f'\in H^1(\R)$, if we replace the $L^1-$ norm of $f'$ by its $H^1-$ norm (see \cite{Oswald}, \cite{K88}).

One of the main results of this paper is the refinement of the inequality (\ref{embed1})   given in terms of
 mixed norms.

Let  $x=(x_1,...,x_n).$  Denote by $\widehat{x}_k$
the $(n-1)-$dimensional vector obtained from the $n$-tuple $x$ by
removal of its $k$th coordinate.
 We shall write
 $x=(x_k,\widehat x_k).$

 If $X(\R)$ and $Y(\R^{n-1})$ are Banach function spaces, and $k\in\{1,...,n\}$, we denote by $Y[X]_k$ the mixed norm space obtained by taking first the norm in $X$ with respect to  $x_k$, and then the norm in $Y$ with respect to $\widehat {x}_k\in \R^{n-1}.$

 We prove the following theorem.
\begin{teo}\label{refined1}
Let $1<p<\infty$ and $n\ge 2$,  or $p=1$ and $n\ge 3.$ If $p<q<\infty$ and $\a=1-(n-1)(1/p-1/q)>0$, then for any $f\in W_p^1(\R^n)$
\begin{equation}\label{embed32}
\sum_{k=1}^n \left(\int_0^\infty h^{-\a p}||\D_k(h)f||^p_{L^{q,p}[L^p]_k}\frac{dh}{h}\right)^{1/p}\le c \sum_{k=1}^n||D_k f||_p.
\end{equation}
\end{teo}

We show that the left-hand side of (\ref{embed1}) is majorized by the left-hand side of (\ref{embed32}). Thus, for the indicated values of $n$ and $p$, Theorem \ref{refined1} provides a refinement of Theorem \ref{1988}. We stress that inequality (\ref{embed32}) holds for $n=2,$  $p>1$. However, {\it the question of the validity
 of  this inequality for $n=2, p=1$ remains open.}

As we have observed above, Theorem \ref{1988} fails for $p=n=1,$ but  in this case there holds a weaker inequality  with $L^1-$norm of $f'$ replaced by  its $H^1-$norm. Similarly, we supplement Theorem \ref{refined1} by the following result.

As usual, for any $1\le p\le \infty$ we denote $p'=p/(p-1).$
\begin{teo}\label{hardy_refined}
Let $f\in W_1^1(\R^n)$ $(n\ge 2)$ and assume that all  partial derivatives $D_jf$ $(j=1,...,n)$ belong to the Hardy space $H^1(\R^n)$. Then for any $1<q<(n-1)/(n-2)$
\begin{equation}\label{embed321}
\sum_{k=1}^n \int_0^\infty h^{(n-1)/q'-1}||\D_k(h)f||_{L^{q,1}[L^1]_k}\frac{dh}{h}\le c \sum_{k=1}^n||D_kf||_{H^1}.
\end{equation}
\end{teo}
That is,
inequality  (\ref{embed32}) holds for $p=1, n=2$ if the $L^1-$norms of the derivatives are replaced by the Hardy $H^1-$norms.
Of course, for $n\ge 3$ (\ref{embed321}) follows from (\ref{embed32}).

We should note that this work was partly inspired by the Oberlin estimate \cite{Ober} of Fourier transforms of functions  in the Hardy space $H^1(\R^n).$
We apply inequality (\ref{embed32}) to obtain an analogue of this estimate for the derivatives of functions in $W_1^1(\R^n)$. In particular, we prove the following result.
\begin{teo}\label{Obertype100} Let $f\in W_1^1(\R^n)$ $(n\ge 3).$ Then
\begin{equation}\label{obertype100}
 \sum_{k\in\Z} 2^{k(2-n)}\sup_{2^{k}\le r\le 2^{k+1}}\int_{S_r} |\widehat f(\xi)| d\s(\xi)\le c ||\nabla f||_1,
\end{equation}
where $S_r$ is the sphere of the radius $r$ centered at the origin in $\R^n$ and $d\s(\xi)$ is the canonical surface measure on $S_r.$
\end{teo}

For $n\ge 3$ this theorem gives a refinement of the Hardy type inequality
\begin{equation*}
\int_{\R^n} |\widehat{f}(\xi)||\xi|^{1-n}\,d\xi\le c
||\nabla f||_1,
\end{equation*}
which was proved for $f\in W_1^1(\R^n) \,\, (n\ge 2)$ by
Bourgain \cite{Bour1}  and  Pe\l czy\'nski and  Wojciechowski
\cite{PeWo}.

As in the case $p=1$ in Theorem \ref{refined1}, {{\it it is an open
question whether Theorem \ref{Obertype100} is true for $n=2.$}}
\vskip 4pt

The paper is organized as follows. We give some definitions and auxiliary results in Section 2.
In Section 3 we prove inequalities between Besov norms built upon the spaces $L^{p,\nu}(\R^n)$ and $L^{p,\nu}(\R^{n-1})[L^r(\R)],$ $1\le r, \nu\le p$. In Section 4 we prove Theorem \ref{refined1}.  Section 5 contains the proof of  Theorem \ref{hardy_refined}.   Section 6 is devoted to estimates of Fourier transforms of functions in $W_1^1(\R^n).$

\section{Some definitions and auxiliary results}

 Denote by $S_0(\mathbb{R}^n)$ the class of all measurable and
almost everywhere finite functions $f$ on $\mathbb{R}^n$ such that
\begin{equation*}
\lambda_f (y) = | \{x \in \mathbb{R}^n : |f(x)|>y \}| <
\infty\quad \text{for each $y>0$}.
\end{equation*}
A nonincreasing rearrangement of a function $f \in
S_0(\mathbb{R}^n)$ is a nonnegative and nonincreasing function $f^*$ on
$\mathbb{R}_+ = (0, + \infty)$ which is equimeasurable with $|f|$, that is, $\lambda_{f^*}=\lambda_f.$
The rearrangement $f^*$ can be defined by the equality
\begin{equation}\label{rearrangement}
f^*(t) = \sup_{|E|=t} \inf_{x \in E} |f(x)|,\quad 0<t<\infty
\end{equation}
(see \cite[p. 32]{ChR}).

The following relation holds \cite[p. 53]{BS}
\begin{equation} \label{supremum}
\sup_{|E|=t} \int_E |f(x)| dx = \int_0^t f^*(u) du \,.
\end{equation}
In what follows we denote
\begin{equation}\label{def**}
f^{**}(t)= \frac{1}{t} \int_0^t f^*(u) du.
\end{equation}
For any $t>0$ there is a subset $E\subset \R^n$ with $|E|=t$ such that
\begin{equation} \label{supremum2}
\frac{1}{t} \int_E|f(x)|dx=f^{**}(t)
\end{equation}
(see \cite[p. 53]{BS}).

Let $0<p,r<\infty.$ A
function $f \in S_0(\mathbb{R}^n)$ belongs to the Lorentz space
$L^{p,r}(\mathbb{R}^n)$ if
\begin{equation*}
\|f\|_{L^{p,r}}=\|f\|_{p,r} = \left( \int_0^\infty \left( t^{1/p} f^*(t)
\right)^r \frac{dt}{t} \right)^{1/r} < \infty.
\end{equation*}

We have that $||f||_{p,p}=||f||_p.$ For a
fixed $p$, the Lorentz spaces $L^{p,r}$ strictly increase as the secondary
index $r$ increases; that is, the strict embedding $L^{p,r}\subset L^{p,s}~~~(r<s)$ holds (see \cite[Ch. 4]{BS}).

\vskip 6pt
We will use the following  Hardy's inequality (see \cite[p. 124]{BS}).

\begin{prop}\label{hardy}
Let $\f$ be a nonnegative
measurable function on $(0,\infty)$ and suppose $-\infty<\lambda<1$ and $1\le p<\infty.$ Then
$$
\biggl(\int_0^{\infty}\Big(t^{\lambda-1}\int_0^t\f(u)du\Big)^p\frac{dt}{t}\biggr)
^{1/p}
\le \frac{1}{1-\lambda}\left(\int_0^\infty\left(t^{\lambda}\f(t)\right)^p\frac{dt}{t}
\right)^{1/p}.
$$
\end{prop}

Applying Hardy's inequality with $p>1,$ $\lambda=1/p$, we obtain that the  operator $f\mapsto f^{**}$  is bounded in $L^p$ for $p>1,$
\begin{equation} \label{bound}
||f^{**}||_p\le \frac{p}{p-1}||f||_p,\quad 1<p\le \infty.
\end{equation}

We say that a measurable function $\psi$ on $(0,\infty)$ is
quasi-decreasing if there exists a constant $c>0$ such that $\psi(t_1)\le c\psi(t_2),$
whenever $0<t_2<t_1<\infty$.

It is well known that in the case $0<p<~1$ Hardy-type inequalities are true for quasi-decreasing
functions. We will use the following proposition
(a short proof can be found, e.g., in \cite{KL}).

\begin{prop}\label{hardy-type}
Let $\psi$ be a non-negative, quasi-decreasing function on
$(0,\infty).$  Suppose also that $\a>0, \b>-1$
and $0<p<1$. Then
$$
\int_0^{\infty}{u}^{-\a-1}\Big(\int_0^{u}\psi(t)t^{\b}dt\Big)^pdu
\le c\int_0^{\infty}
u^{-\a-1}\big(\psi(u)u^{\b+1}\big)^pdu.
$$
\end{prop}

Let a function $\f\in L^p(\R)$. Set
\begin{equation}\label{one}
\Delta(h)\f(x)=\f(x+h)-\f(x), \quad h\in \R,
\end{equation}
and
$$
\o(\f;t)_p=\sup_{|h|\le t}||\Delta(h)\f||_p, \quad t\ge 0.
$$
 Ul'yanov \cite{U} proved the following estimate: for any $\f\in L^p(\R),\,1\le p <\infty$
 $$
 \f^{**}(t)-\f(t)\le 2t^{-1/p}\o(\f;t)_p.
 $$
 It easily follows that
\begin{equation}\label{Ulyanov1}
\f^*(t)\le 2\int_t^\infty s^{-1/p}\o(\f;s)_p\frac{ds}{s}
\end{equation}
(see also \cite[p. 149]{K98}, \cite{Stor}).
Using these estimates, Ul'yanov obtained that if $1\le p <q<\infty$ and  $\f\in L^p(\R)$, then
\begin{equation}\label{Ulyanov2}
||\f||_q\le c \left(\int_0^\infty t^{-q/p}||\Delta(t)\f||_p^q\,dt\right)^{1/q}
\end{equation}
and
\begin{equation}\label{Ulyanov3}
\o(\f;\d)_q\le c\left(\int_0^\d t^{-q/p}||\Delta(t)\f||_p^q\,dt\right)^{1/q}
\end{equation}
(some discussions and generalizations of these results can be found in    \cite{K98} and \cite{K07}).

In the next section we consider functions $(x,y)\mapsto f(x,y),$ where
 $x\in \R, \,\, y\in \R^{n-1},$ and we denote
\begin{equation}\label{several}
\Delta_1(h)f(x,y)=f(x+h,y)-f(x,y), \quad h\in \R.
\end{equation}
Let $V=V(\R^n)$ be a Banach function space over $\R^n$ (see \cite[Ch. 1]{BS}). We shall assume that $V$ is translation invariant, that is,
whenever $f\in V,$ then $\tau_h f\in V$ and $||\tau_h f||_V=||f||_V$ for all $h\in \R^n$, where
$\tau_h f(x)=f(x-h).$
Let $f\in V$. Set
$$
\o_1(f;\d)_V=\sup_{|h|\le \d}||\Delta_1(h)f||_V, \quad \d\ge 0.
$$
In these notations, the  subindex 1 indicates that the difference is taken with respect to the first variable $x.$

We have the following inequality
\begin{equation}\label{omega1}
\o_1(f;\delta)_V\le \frac{3}{\d}\int_0^\d ||\Delta_1(h)f||_Vdh.
\end{equation}
Indeed, if $t, h\in [0,\d],$ then
$$
||\Delta_1(t)f||_V\le ||\Delta_1(h)f||_V + ||\Delta_1(t-h)f||_V.
$$
Integrating with respect to $h$ in $[0,\d]$ (for a fixed $t\in [0,\d]$),  and then taking supremum over $t,$ we obtain (\ref{omega1}).

\vskip 6pt

\section{Different norm inequalities}

\vskip 6pt

Throughout this paper we use the notation (\ref{Delta}).

Let $0<\a<1,$  $1\le p<\infty$, and $1\le \t<\infty.$
 The Besov space
$B^\a_{p,\t}(\mathbb R^n)$ consists of all functions $f\in
L^p(\mathbb R^n)$ such that
$$
\|f\|_{B^\a_{p,\t}}=||f||_p+
\sum_{k=1}^n\Big(\int_0^{\infty}\big(t^{-\a}||\Delta_k(t)f||_p\big)^\t\,
\frac{dt}{t}\Big)^{1/\t}<\infty.
$$

\vskip 6pt

The classical different norm embedding theorem states that if $1\le p<q<\infty$ and $\a>n(1/p-1/q),$ then for any $1\le\t<\infty$
$$
B_{p,\t}^\a(\R^n)\subset B_{q,\t}^\b(\R^n),\quad\mbox{where} \quad \b= \a-n(1/p-1/q),
$$
and for any $f\in B_{p,\t}^\a(\R^n)$
\begin{equation}\label{diff}
||f||_{B^\b_{q,\t}}\le c||f||_{B^\a_{p,\t}}
\end{equation}
(see \cite[Ch. 6]{Nik}).

Roughly speaking, passing from $L^p$ to $L^q$, we lose $n(1/p-1/q)$ in the  smoothness exponent.

We shall be especially interested in the one-dimensional case of this theorem.
Note that   for $n=1$ (\ref{diff}) follows immediately from (\ref{Ulyanov2}), (\ref{Ulyanov3}) and Hardy's inequality.

In this section we obtain different norm inequalities for the Besov spaces defined in some mixed norms. First of all, we are interested in these results in connection with embeddings of Sobolev spaces (in particular, for the comparison of Theorems \ref{refined1} and \ref{1988}).

We keep  notations introduced in Section 2. Namely, we use the notation $\Delta(h)\f$ for functions of one variable (see (\ref{one})). The  notation $\Delta_1(h)f$  (see (\ref{several})) is applied to functions $(x,y)\mapsto f(x,y),$ where
 $x\in \R, \,\, y\in \R^{n-1}\,\,(n\ge 2)$.

Let $1\le \t<\infty, \,\, 0<\a<1.$ Let $V=V(\R^n)\,\,(n\ge 2)$ be a translation invariant Banach function space. Denote by $B^\a_{\t;1}(V)$ the class of all functions $f\in V$ such that
$$
||f||_{B^\a_{\t;1}(V)}= ||f||_V+\left(\int_0^\infty[h^{-\a}\o_1(f;h)_V]^\t \frac{dh}{h}\right)^{1/\t}<\infty.
$$
As above, the  subindex 1 indicates that the difference is taken with respect to the first variable $x.$  Applying (\ref{omega1}) and Hardy's inequality, we obtain that
\begin{equation}\label{equivalence}
\int_0^\infty[h^{-\a}\o_1(f;h)_V]^\t \frac{dh}{h}\le c\int_0^\infty[h^{-\a}||\Delta_1(h)f||_V]^\t \frac{dh}{h}.
\end{equation}

As in Introduction, if $X(\R)$ and $Y(\R^{n-1})$ are Banach function spaces, we denote by $Y[X]_1$ the mixed norm space obtained by taking first the norm in $X(\R)$ with respect to the variable $x$, and then the norm in $Y(\R^{n-1})$ with respect to $y.$
In this section the interior norm will be taken only in variable $x.$ Therefore  in this section we  write simply $Y[X]$ (omitting the subindex 1).

First, we have the following simple proposition.
\begin{prop}\label{simple} Let $1\le\t<\infty,$ $1\le r<p<\infty$, and $1/r-1/p<\a<1$. Set $\b=\a-1/r+1/p.$ Then
$B^\a_{\t;1}(L^p[L^r])\subset B^\b _{\t;1}(L^p(\R^n))$; more exactly, for any $f\in B^\a_{\t;1}(L^p[L^r])$
\begin{equation}\label{simple1}
||f||_p\le c ||f||_{B^\a_{\t;1}(L^p[L^r])}
\end{equation}
and
\begin{equation}\label{simple2}
\int_0^\infty h^{-\t\b}||\Delta_1 (h)f||_p^\t\frac{dh}{h} \le c\int_0^\infty h^{-\t\a} ||\Delta_1(h)f||^\t_{L^p[L^r]}\frac{dh}{h}.
 \end{equation}
\end{prop}
\begin{proof} Denote $V=L^p[L^r].$ Let  $f\in B^\a_{\t;1}(V)$.
For a fixed $y\in \R^{n-1}$, set $f_y(x)=f(x,y),\,\,x\in \R.$
By (\ref{Ulyanov2}), we have
$$
||f_y||_p^p\le c \int_0^\infty t^{-p/r}||\Delta(t)f_y||_r^pdt.
$$
Integration with respect to $y$ gives
$$
||f||_p^p\le c \int_0^\infty t^{-p/r}||\Delta_1(t)f||_V^pdt.
$$
Applying standard reasonings (see, e.g., \cite[Ch. 5.4]{BS}), we get
$$
\begin{aligned}
&\left(\int_0^\infty t^{-p/r}||\Delta_1(t)f||_{V}^pdt\right)^{1/p}\\
&\le c \left[||f||_{V}+\left(\int_0^1[t^{-\t\a}||\Delta_1(t)f||_{V}]^\t \frac{dt}{t}\right)^{1/\t}\right].
\end{aligned}
$$
These estimates imply (\ref{simple1}).

Further, inequality (\ref{Ulyanov3}) gives that
$$
||\Delta (h)f_y||_p^p\le c\int_0^h ||\Delta (t)f_y||_r^p t^{-p/r}dt.
$$
Integrating with respect to $y$, we get
$$
\int_{\R^{n}}|\Delta_1 (h)f(x,y)|^p(x,y)dxdy= \int_{\R^{n-1}}||\Delta (h)f_y||_p^pdy
$$
$$
\le c \int_{\R^{n-1}}\int_0^h ||\Delta (t)f_y||_r^p t^{-p/r}dtdy=c\int_0^h ||\Delta_1(t)f||^p_{V}t^{-p/r}dt.
$$
This implies that
$$
\int_0^\infty h^{-\t\b}||\Delta_1 (h)f||_p^\t
 \frac{dh}{h}
 $$
 $$
 \le c \int_0^\infty h^{-\t\b}\left(\int_0^h ||\Delta_1(t)f||^p_{V}t^{-p/r}dt\right)^{\t/p}
 \frac{dh}{h}.
 $$
 If $\t\ge p$, then we apply Proposition \ref{hardy} and we obtain (\ref{simple2}).
 Let $\t<p$. Observe that the function
 $\psi(t)=\o_1(f;t)_V/t$
is quasi-decreasing. Hence, applying
  Proposition  \ref{hardy-type} and inequality (\ref{equivalence}), we get
$$
\begin{aligned}
 &\int_0^\infty h^{-\t\b}||\Delta_1 (h)f||_p^\t
 \frac{dh}{h}\\
&\le c \int_0^\infty h^{-\t\b}\left(\int_0^h \o_1(f;t)^p_{V}t^{-p/r}dt\right)^{\t/p}
 \frac{dh}{h}\\
&\le c'\int_0^\infty h^{-\t\a}\o_1(f;h)^\t_{V}\frac{dh}{h}\le c''\int_0^\infty h^{-\t\a}||\Delta_1(h)f||^\t_{V}\frac{dh}{h}.
\end{aligned}
$$
This implies (\ref{simple2}).
 \end{proof}

Note that,  in contrast to (\ref{diff}), the loss in the smoothness exponent given by (\ref{simple2}) is only $1/r-1/p.$ It is natural because the integrability exponent changes in only one variable.

Now, we replace the $L^p$-norm  in (\ref{simple1}) and (\ref{simple2}) by the $L^{p,\nu}-$Lorentz norm.  In this case simple arguments similar
to those given above cannot be applied. Indeed, it was shown by Cwikel \cite{Cwikel} that if $p\not = \nu,$ then neither  of the spaces
$L^{p,\nu}(\R^2)$ and $L^{p,\nu}(\R)[L^{p,\nu}(\R)]$ is contained in the other. Therefore we apply different methods; namely, we shall use iterated rearrangements.

Let $g\in S_0(\R^n), \,\, n\ge 2.$ For a fixed $y\in\R^{n-1},$ denote by $\mathcal R_1 g(s,y)$ the nonincreasing rearrangement of the function $g_y(x)=g(x,y), \,\, x\in \R.$ Further, for a fixed $s>0$, let $\mathcal R_{1,2}g(s,t)$ be the nonincreasing rearrangement of the function $y\mapsto \mathcal R_1 g(s,y),\,\,y\in\R^{n-1}.$

    The iterated rearrangement
$\mathcal R_{1,2} g$  is defined on $\R_+^2.$ It is nonnegative, nonincreasing in each variable,
and equimeasurable
with  $|g|$ function (see \cite{Bl, K01, K07}).

 Let $0<p, \nu<\infty,$ and $n\ge 2$.  For a function $g\in S_0(\R^n),$ denote
$$
\| g\| _{\mathcal L^{p,\nu}}=\left(\int_{\Bbb R _+^2}
(st)^{\nu/p-1}\mathcal R_{1,2}g(s,t)^\nu\,ds
dt\right) ^{1/\nu}
$$
(see \cite{Bl}).
The following inequalities hold \cite{Ya}:
\begin{equation}\label{const1}
\| g\| _{p,\nu}\le  c\| g\|
_{\mathcal L^{p,\nu}}\quad\text{if}\quad 0<\nu\le p<\infty
\end{equation}
and
\begin{equation}\label{const2}
\| g\| _{\mathcal L^{p,\nu}}\le  c'\| g\|
_{p,\nu}\quad\text{if}\quad 0<p\le \nu<\infty.
\end{equation}

\begin{prop}\label{Strong} Let $1\le\t<\infty,$ $1\le \nu\le p<\infty$, $1\le r<p$, and $1/r-1/p<\a<1$.
Set $\b=\a-1/r+1/p.$ Then
$B^\a_{\t;1}(L^{p,\nu}[L^r])\subset B^\b_{\t;1}(L^{p,\nu})$; more exactly, for any $f\in B^\a_{\t;1}(L^{p,\nu}[L^r])$
\begin{equation}\label{strong1}
||f||_{L^{p,\nu}}\le c||f||_{B^\a_{\t;1}(L^{p,\nu}[L^r])}
\end{equation}
and
\begin{equation}\label{strong10}
\int_0^\infty h^{-\t\b}||\Delta_1 (h)f||_{L^{p,\nu}}^\t\frac{dh}{h}\le c\int_0^\infty h^{-\t\a}||\Delta_1(h)f||_{L^{p,\nu}[L^r]}^\t \frac{dh}{h}.
\end{equation}
\end{prop}
\begin{proof}
Let $f\in B_{\t,1}^\a(L^{p,\nu}[L^r]).$ Set $\f_h(x,y)=|\Delta_1 (h)f(x,y)|.$ Let $s$ and $h$ be fixed positive numbers. We consider the function $y\mapsto \mathcal R_1 \f_h(s,y),\,\, y\in \R^{n-1}.$ As in Section 2 above (see (\ref{supremum2})), we can state that for any $t>0$  there exists a set $E=E_{s,t,h}\subset \R^{n-1}$ with $\mes_{n-1}E=t$ such that
\begin{equation}\label{strong23}
\mathcal R_{1,2}\f_h(s,t)\le \frac{1}{t} \int_{E} \mathcal R_1 \f_h(s,y)\,dy.
\end{equation}
By (\ref{Ulyanov1}), for any $s>0$
\begin{equation}\label{strong2}
\mathcal R_1 \f_h(s,y)\le 2\int_s^\infty \o(\f_h(\cdot,y);u)_r\frac{du}{u^{1+1/r}}.
\end{equation}
Set $g_{u,h}(y)= \o(\f_h(\cdot,y);u)_r.$ By (\ref{supremum}), we have
\begin{equation}\label{strong24}
\frac1t\int_{E} g_{u,h}(y)\,dy\le g_{u,h}^{**}(t).
\end{equation}
Applying inequalities (\ref{strong23}), (\ref{strong2}), and (\ref{strong24}), we obtain
$$
\mathcal R_{1,2}\f_h(s,t)\le \frac{2}{t}\int_s^\infty \int_{E} g_{u,h}(y)\,dy\frac{du}{u^{1+1/r}}
$$
$$
\le 2\int_s^\infty g_{u,h}^{**}(t)\frac{du}{u^{1+1/r}}.
$$

Further, we shall estimate
$$
||\Delta_1 (h)f||_{\mathcal L^{p,\nu}}^\nu=\int_0^\infty\int_0^\infty(st)^{\nu/p-1}
\mathcal R_{1,2}\f_h(s,t)^\nu dsdt.
$$
Fix $t>0$. Applying Hardy's inequality, we have
$$
\int_0^\infty s^{\nu/p-1}
\mathcal R_{1,2}\f_h(s,t)^\nu ds\le 2^\nu \int_0^\infty s^{\nu/p-1} \left(\int_s^\infty g_{u,h}^{**}(t)\frac{du}{u^{1+1/r}}\right)^\nu ds
$$
$$
\le c\int_0^\infty s^{\nu/p-\nu/r-1}g_{s,h}^{**}(t)^\nu ds.
$$
Thus,
$$
\begin{aligned}
&||\Delta_1 (h)f||_{\mathcal L^{p,\nu}}^\nu=\int_{\Bbb R _+^2}
(st)^{\nu/p-1}\mathcal R_{1,2}\f_h(s,t)^\nu\,dsdt\\
&\le c\int_0^\infty s^{\nu/p-\nu/r-1}
\int_0^\infty t^{\nu/p-1}g_{s,h}^{**}(t)^\nu dtds\\
&\le  c'\int_0^\infty s^{\nu/p-\nu/r-1}||g_{s,h}||_{L^{p,\nu}}^\nu ds.
\end{aligned}
$$
By (\ref{omega1}), we have
$$
g_{s,h}(y)= \o(\f_h(\cdot,y);s)_r\le \frac{c}{s}\int_0^s ||\D(u)\f_h(\cdot,y)||_rdu.
$$
Thus, by the Minkowski inequality,
$$
||g_{s,h}||_{L^{p,\nu}}\le \frac{c}{s}\int_0^s ||\Delta_1(u)\f_h||_V du, \quad\mbox{where}\quad V=L^{p,\nu}[L^r].
$$
Using this estimate and applying Hardy's inequality, we obtain
$$
||\Delta_1 (h)f||_{\mathcal L^{p,\nu}}^\nu\le c \int_0^\infty s^{\nu/p-\nu/r-1}\left(\frac{1}{s}\int_0^s ||\Delta_1(u)\f_h||_V du\right)^\nu ds
$$
$$
\le c'\int_0^\infty s^{\nu/p-\nu/r-1}||\Delta_1(s)\f_h||_V^\nu ds.
$$
Obviously,
$$
||\D_1(s)\f_h||_V\le 2||\D_1(\min(s,h))f||_V.
$$
Thus,
$$
\begin{aligned}
&\int_0^\infty h^{-\t\b}||\Delta_1 (h)f||_{\mathcal L^{p,\nu}}^\t\frac{dh}{h}\\
&\le c \int_0^\infty h^{-\t\b}\left(\int_0^\infty s^{\nu/p-\nu/r-1}||\D_1(\min(s,h))f||_V^\nu ds\right)^{\t/\nu}\frac{dh}{h}\\
&\le c'\left[ \int_0^\infty h^{-\t\b}\left(\int_0^h s^{\nu/p-\nu/r-1}||\D_1(s)f||_V^\nu ds\right)^{\t/\nu}\frac{dh}{h}\right.\\
&+\left.\int_0^\infty h^{-\t\b}||\D_1(h)f||_V^\t\left(\int_h^\infty s^{\nu/p-\nu/r-1} ds\right)^{\t/\nu}\frac{dh}{h}\right]\equiv c(I_1+I_2).
\end{aligned}
$$
First,
\begin{equation}\label{strong1000}
I_2=c\int_0^\infty h^{-\t\a}||\D_1(h)f||_V^\t \frac{dh}{h}.
\end{equation}
Further, if  $\t>\nu$,  then by Proposition \ref{hardy}   we obtain
\begin{equation}\label{strong10000}
I_1\le c\int_0^\infty h^{-\t\a}||\D_1(h)f||_V^\t \frac{dh}{h}.
\end{equation}
If $\t\le\nu,$ we obtain estimate (\ref{strong10000}) exactly as in Proposition \ref{simple}. Namely, using the fact that the function
 $\psi(t)=\o_1(f;t)_V/t$
is quasi-decreasing, we apply
  Proposition  \ref{hardy-type} and inequality (\ref{equivalence}).
Estimates (\ref{strong1000}) and (\ref{strong10000}) give that
$$
\int_0^\infty h^{-\t\b}||\Delta_1 (h)f||_{\mathcal L^{p,\nu}}^\t\frac{dh}{h}\le c\int_0^\infty h^{-\t\a}||\D_1(h)f||_V^\t \frac{dh}{h}.
$$
Since $\nu\le p$,  the latter inequality implies
 (\ref{strong10}) (see (\ref{const1})).

 Inequality (\ref{strong1}) follows by similar arguments; we omit the details.

\end{proof}

\vskip 6pt
\begin{rem} In this work we apply Proposition \ref{Strong} only for $\nu=r<p$. It would be interesting to consider other cases and further generalizations in this direction.

\end{rem}

\section{ Embeddings of  Sobolev spaces $W_p^1(\R^n)$}

In this section we prove a refinement of Theorem \ref{1988}.
For $1\le p,q<\infty$ and $k=1,...,n$,  denote by  $V_{q,p,k}(\R^n)$ the mixed norm space
$L^{q,p}(\R^{n-1})[L^p(\R)]_k$
obtained by taking first the norm in $L^p(\R)$ with respect to the variable $x_k$, and then the norm in $L^{q,p}(\R^{n-1})$ with respect to $\widehat x_k.$

We shall use the following simple fact.
\begin{prop}\label{absolute}
Let a function $\f$ be defined on $\R$ and assume that $\f$ is locally absolutely continuous (that is, $\f$ is absolutely continuous in each bounded interval $[a,b]\subset \R).$ Let $\psi=|\f|$. Then $\psi$ also is locally absolutely continuous and
$$
|\psi'(x)|\le |\f'(x)| \quad\mbox{for almost}\quad x\in \R.
$$
\end{prop}

Indeed, this statement follows immediately from the inequality
$$
|\psi(x+h)-\psi(x)|\le |\f(x+h)-\f(x)|.
$$

\begin{teo}\label{refined}
Let $1<p<\infty$ and $n\ge 2,$  or $p=1$ and $n\ge 3.$ If $p<q<\infty$ and $\a=1-(n-1)(1/p-1/q)>0$, then for any $f\in W_p^1(\R^n)$
\begin{equation}\label{embed3}
\sum_{k=1}^n \left(\int_0^\infty h^{-\a p}||\D_k(h)f||^p_{V_{q, p,k}}\frac{dh}{h}\right)^{1/p}\le c ||\nabla f||_p.
\end{equation}
\end{teo}
\begin{proof}
We estimate the last term of the sum in (\ref{embed3}). Set
$$
\f_h(\widehat x_n)=\left(\int_\R |\D_n(h)f(x)|^pdx_n\right)^{1/p}
$$
and
$$
\psi_j(\widehat x_n)=\left(\int_\R |D_j f(x)|^pdx_n\right)^{1/p},\quad j=1,...,n.
$$
We consider the integral
\begin{equation}\label{embed5}
J=\int_0^\infty h^{-\a p}K(h)\frac{dh}{h},
\end{equation}
where
$$
K(h)=||\D_n(h)f||_{V_{q,p,n}}^p=\int_0^\infty t^{p/q-1}\f_h^*(t)^pdt.
$$
Set
\begin{equation}\label{embed100}
K_1(h)=\int_{h^{n-1}}^\infty t^{p/q-1}\f_h^*(t)^pdt, \,\,\, K_2(h)=\int_0^{h^{n-1}} t^{p/q-1}\f_h^*(t)^pdt.
\end{equation}
For any $h>0$
$$
|\D_n(h)f(x)|\le \int_0^h |D_n f(x+u e_n)| du.
$$
Raising to the power $p$, integrating over $x_n$ in $\R,$ and applying H\"older's inequality, we obtain
$$
\f_h(\widehat x_n)^p\le \int_\R \left(\int_0^h |D_n f(x+u e_n)| du \right)^p d x_n\le h^p\psi_n(\widehat x_n)^p.
$$
Thus,
\begin{equation}\label{embed101}
\f_h^*(t)\le h\psi_n^*(t).
\end{equation}
From here (see (\ref{embed100}))
$$
K_1(h)\le h^p\int_{h^{n-1}}^\infty t^{p/q-1}\psi_n^*(t)^p dt
$$
and therefore
$$
\begin{aligned}
J_1&= \int_0^\infty h^{-\a p}K_1(h)\frac{dh}{h}\le \int_0^\infty h^{(1-\a)p}\int_{h^{n-1}}^\infty t^{p/q-1}\psi_n^*(t)^p dt\frac{dh}{h}\\
&=\int_0^\infty t^{p/q-1}\psi_n^*(t)^p \int_0^{t^{1/(n-1)}} h^{(1-\a) p}\frac{dh}{h}dt\\
&=((1-\a)p)^{-1}\int_0^\infty \psi_n^*(t)^p dt=c||D_nf||_p^p.
\end{aligned}
$$
This estimate holds for all $p\ge 1$ and $n\ge 2.$

Estimating $K_2(h)$, we first assume that $p=1$ and $n\ge 3.$
Set
$$
g(\widehat x_n)=\int_\R |f(x)|\,dx_n.
$$
Then $||g||_{L^1(\R^{n-1})}=||f||_{L^1(\R^n)}.$ Moreover, $g\in W_1^1(\R^{n-1})$ and
\begin{equation}\label{embed4}
||D_j g||_{L^1(\R^{n-1})}\le ||D_{j} f||_{L^1(\R^n)},\quad j=1,...,n-1.
\end{equation}
Indeed, since $f\in W_p^1(\R^n)$, then for any $j=1,...,n$ and almost all $\widehat x_j\in\R^{n-1}$ the function $f$ is locally absolutely continuous with respect to $x_j$ (see, e.g., \cite[2.1.4]{Zie}). Thus, we can apply Proposition \ref{absolute}.

We have
$$
\f_h(\widehat x_n)\le \int_\R |f(x)|dx_n+\int_\R |f(x+h e_n)|dx_n = 2g(\widehat x_n).
$$
Thus (see (\ref{embed100})),
$$
K_2(h)\le 2\int_0^{h^{n-1}} t^{1/q-1}g^*(t)dt
$$
and
$$
\begin{aligned}
J_2&= \int_0^\infty h^{-\a}K_2(h)\frac{dh}{h}\le 2\int_0^\infty  h^{-\a}\int_0^{h^{n-1}} t^{1/q-1}g^*(t)dt\frac{dh}{h}\\
&=2\int_0^\infty t^{1/q-1}g^*(t)\int_{t^{1/(n-1)}}^\infty h^{(1-1/q)(n-1)-1}\frac{dh}{h}\\
&=c\int_0^\infty t^{-1/(n-1)}g^*(t)dt=c||g||_{(n-1)',1}.
\end{aligned}
$$
Taking into account (\ref{embed4}) and applying inequality (\ref{embed0}), we get
\begin{equation*}
J_2\le c||g||_{(n-1)',1}\le c'\sum_{j=1}^{n-1}||D_{j}f||_1.
\end{equation*}
Together with the estimate $J_1\le c ||D_nf||_1$ obtained above, this gives (\ref{embed3}) for $p=1, n\ge 3.$

Let now $p>1,$ $n\ge 2.$  In what follows we write $x=(u,x_n),\, u=\widehat x_n\in \R^{n-1}.$

For a fixed $u\in \R^{n-1}$ and $t>0,$ denote by $Q_u(t)$ the cube in $\R^{n-1}$ centered at $u$ with the side length $(4t)^{1/(n-1)}.$
Let
$$
A_{u,t,h} =\{v\in Q_u(t): \f_h(v)\le \f_h^*(2t)\}.
$$
Then $\mes_{n-1}A_{u,t,h}\ge 2t.$ Thus, we have
$$
\f_h(u)-\f_h^*(2t)\le \f_h(u)-\frac{1}{\mes_{n-1}A_{u,t,h}}\int_{A_{u,t,h}}\f_h(v)dv
$$
\begin{equation}\label{nuevo}
\le \frac1{2t}\int_{Q_u(t)} |\f_h(u)-\f_h(v)|dv.
\end{equation}
Further,
$$
|\f_h(u)-\f_h(v)|=\left|\left(\int_\R|f(u,x_n+h)-f(u,x_n)|^pdx_n\right)^{1/p}\right.
$$
$$
-\left.\left(\int_\R|f(v,x_n+h)-f(v,x_n)|^pdx_n\right)^{1/p}\right|
$$
$$
\le 2\left(\int_\R|f(u,x_n)-f(v,x_n)|^pdx_n\right)^{1/p}.
$$
We have (see \cite[p. 143]{LL})
$$
|f(u,x_n)-f(v,x_n)|\le|u-v|\sum_{j=1}^{n-1}\int_0^1|D_jf(u+\tau(v-u), x_n)|d\tau.
$$
If $v\in Q_u(t),$ then $|u-v|\le \sqrt{n-1}(2t)^{1/(n-1)}.$ Thus, by the Minkowski inequality, for any $v\in Q_u(t)$
$$
\begin{aligned}
&|\f_h(u)-\f_h(v)|\\
&\le c t^{1/(n-1)}\sum_{j=1}^{n-1}\left(\int_\R\left(\int_0^1|D_jf(u+\tau(v-u), x_n)|d\tau\right)^pdx_n\right)^{1/p}\\
&\le ct^{1/(n-1)}\sum_{j=1}^{n-1}\int_0^1\left(\int_\R|D_jf(u+\tau(v-u), x_n)|^pdx_n\right)^{1/p}d\tau\\
&=ct^{1/(n-1)}\sum_{j=1}^{n-1}\int_0^1\psi_j(u+\tau(v-u))d\tau.
\end{aligned}
$$
From here and (\ref{nuevo}),
\begin{equation}\label{embed6}
\f_h(u)-\f_h^*(2t)\le ct^{1/(n-1)-1}\sum_{j=1}^{n-1}\int_{Q_0(t)}\int_0^1\psi_j(u+\tau z)d\tau dz.
\end{equation}
Taking into account that
$$
\f_h^*(t)\le \sup_{\mes_{n-1}E=t}\frac1t\int_E\f_h(u)du,
$$
and applying
 (\ref{embed6}), we get
 $$
 \f_h^*(t)-\f_h^*(2t)\le \sup_{\mes_{n-1}E=t}\frac1t\int_E[\f_h(u)-\f_h^*(2t)]du
$$
$$
\le  ct^{1/(n-1)-1}\sum_{j=1}^{n-1}\sup_{\mes_{n-1}E=t}\int_{Q_0(t)}\int_0^1\frac1t\int_E\psi_j(u+\tau z)dud\tau dz.
$$ 
Let $E\subset \R^{n-1}$, $\mes_{n-1} E=t$. Then for all $\tau\in [0,1]$ and $z\in Q_0(t)$
$$
\frac1t\int_E \psi_j(u+\tau z)du\le \psi^{**}_j(t).
$$
Thus, we have that
\begin{equation}\label{embed7}
\f_h^*(t)-\f_h^*(2t)\le ct^{1/(n-1)}\sum_{j=1}^{n-1}\psi^{**}_j(t).
\end{equation}

Now, for any $\e>0,$ we have
$$
J_2(\e)^{1/p}= \left(\int_\e^{1/\e} h^{-\a p}\int_{\e^{n-1}}^{h^{n-1}}t^{p/q-1}\f_h^*(t)^pdt\frac{dh}{h}\right)^{1/p}
$$
$$
\le \left(\int_0^\infty h^{-\a p}\int_0^{h^{n-1}}t^{p/q-1}[\f_h^*(t)-\f_h^*(2t)]^pdt\frac{dh}{h}\right)^{1/p}
$$
$$
+\left(\int_\e^{1/\e} h^{-\a p}\int_{\e^{n-1}}^{h^{n-1}}t^{p/q-1}\f_h^*(2t)^pdt\frac{dh}{h}\right)^{1/p}\equiv I'+I''(\e).
$$
By (\ref{embed7}) and (\ref{bound}),
$$
I'\le c\sum_{j=1}^{n-1}\left(\int_0^\infty h^{-\a p}\int_0^{h^{n-1}}t^{p/q+p/(n-1)-1}\psi^{**}_j(t)^pdt\frac{dh}{h}\right)^{1/p}
$$
$$
=c\sum_{j=1}^{n-1}\left(\int_0^\infty t^{p/q+p/(n-1)-1}\psi^{**}_j(t)^p\int_{t^{1/(n-1)}}^\infty h^{-\a p}\frac{dh}{h}dt\right)^{1/p}
$$
$$
=c'\sum_{j=1}^{n-1}\left(\int_0^\infty\psi^{**}_j(t)^p dt\right)^{1/p}\le c''\sum_{j=1}^{n-1}||\psi_j||_p=c''\sum_{j=1}^{n-1}||D_jf||_p.
$$
Further,
$$
I''(\e)=\left(2^{-p/q}\int_\e^{1/\e} h^{-\a p}\int_{2\e^{n-1}}^{2h^{n-1}}t^{p/q-1}\f_h^*(t)^pdt\frac{dh}{h}\right)^{1/p}
$$
$$
\le 2^{-1/q}\left(\int_\e^{1/\e} h^{-\a p}\int_{\e^{n-1}}^{h^{n-1}}t^{p/q-1}\f_h^*(t)^pdt\frac{dh}{h}\right)^{1/p}
$$
$$
+2^{-1/q}\left(\int_0^\infty h^{-\a p}\int_{h^{n-1}}^\infty t^{p/q-1}\f_h^*(t)^pdt\frac{dh}{h}\right)^{1/p}
$$
$$
=2^{-1/q}\left(J_2(\e)^{1/p}+J_1^{1/p}\right).
$$
As we have proved above, $J_1^{1/p}\le c ||D_nf||_p$. Thus,
$$
J_2(\e)^{1/p}\le I'+I''(\e)\le 2^{-1/q}J_2(\e)^{1/p}+ c\sum_{j=1}^{n}||D_jf||_p
$$
and
$$
J_2(\e)^{1/p}\le c'\sum_{j=1}^{n}||D_jf||_p.
$$
This implies that
$$
J_2^{1/p}=\left(\int_0^\infty h^{-\a p}\int_0^{h^{n-1}}t^{p/q-1}\f_h^*(t)^pdt\frac{dh}{h}\right)^{1/p}\le c'\sum_{j=1}^{n}||D_jf||_p.
$$
Thus, we have (see notations (\ref{embed5}) and (\ref{embed100}))
$$
J^{1/p}\le J_1^{1/p}+J_2^{1/p}\le c''\sum_{j=1}^{n}||D_jf||_p.
$$
In turn, this yields (\ref{embed3}) for $p>1,\,n\ge 2.$
\end{proof}

\begin{rem} By Proposition \ref{Strong}, inequality (\ref{embed3}) gives a refinement of (\ref{embed1}).

We stress that (\ref{embed3})  is true  for $p>1, n=2$. As it was already observed,  {\it we do not know whether this inequality  is true for $p=1, \,n=2.$} However, we shall show that similar  inequality holds for $p=1, \,n=2$ if we replace
the $L^1-$norms of derivatives  by the Hardy $H^1-$norms.
\end{rem}

\vskip 6pt

\vskip 6pt

\section{ Embeddings of  Hardy-Sobolev spaces}

\vskip 6pt

For a function $f\in L^1(\R^n)$ the Fourier transform
is defined by
$$
\widehat{f}(\xi)=\int_{\R^n}f(x)e^{-i2\pi x\cdot \xi}\,dx, \quad
\xi\in \R^n.
$$

Let $f\in L^1(\R^n).$ The Riesz transforms $R_jf\,\, (j=1,...,n)$ of $f$ are defined by the equality
$$
R_jf(x)=\lim_{\e\to 0+} c_n\int_{|y|\ge \e} \frac{y_j}{|y|^{n+1}}f(x-y)dy,\quad  c_n=\frac{\Gamma((n+1)/2)}{\pi^{(n+1)/2}}.
$$

The space $H^1(\R^n)$ is  the class of all functions $f\in L^1(\R^n)$ such that $R_jf\in L^1(\R^n)$ $(j=1,...,n)$.
The $H^1-$norm is defined by
$$
||f||_{H^1}=||f||_1+\sum_{j=1}^n||R_jf||_1
$$
(see \cite[p. 144]{FS}, \cite[Ch. III.4]{GR}).

If $f\in H^1(\R^n)$, then we have  (see \cite[p. 197]{GR})
$$
(R_jf)^\land(\xi)=-\frac{i\xi_j}{|\xi|}\widehat f(\xi).
$$

Let $P_t$ be the Poisson kernel in $\R^n$. We consider $n+1$ harmonic functions in $\R^{n+1}_+=\R^n\times (0,+\infty)$
\begin{equation}\label{harm1}
u_0(x,t)=(P_t\ast f)(x), \,\, u_j(x,t)=(P_t\ast R_jf)(x) \,\,(j=1,...,n).
\end{equation}
These functions satisfy the equations of conjugacy
\begin{equation}\label{conj}
\frac{\partial u_j}{\partial x_k}=\frac{\partial u_k}{\partial x_j}\quad(0\le j,k\le n), \quad
\sum_{j=0}^n\frac{\partial u_j}{\partial x_j}=0 \quad (x_0=t)
\end{equation}
(see \cite[Ch. III.4]{GR}).

For any $x\in \R^n$, denote by $\Gamma(x)$ the cone
$$
\Gamma(x)=\{(y,t)\in \R^{n+1}_+: |x-y|\le t\}.
$$

Let $f\in L^1(\R^n)$. The non-tangential maximal function $Nf$ is defined by
$$
Nf(x)=\sup_{(y,t)\in \Gamma(x)} |(P_t\ast f)(y)|.
$$
A function $f\in L^1(\R^n)$ belongs to $H^1(\R^n)$ if and only if $Nf\in L^1(\R^n)$.
In this case
\begin{equation}\label{max}
c||f||_{H^1}\le ||Nf||_1\le c'||f||_{H^1}\quad (c>0)
\end{equation}
(see \cite[Ch. III.4]{GR}, \cite[Th. 6.7.4]{Graf}).

The non-tangential maximal function $Nf$ is controlled by the vertical maximal function
$$
N_vf(x)=\sup_{t>0}|(P_t\ast f)(x)|.
$$
Namely, $Nf\in L^1(\R^n)$ if and only if $N_vf\in L^1(\R^n)$, and in this case
\begin{equation}\label{vert}
||N_vf||_1\le ||Nf||_1\le c||N_vf||_1
\end{equation}
(see  \cite[p.170]{FS}, \cite[Th. 6.4.4]{Graf}).

Furthermore, if $f\in H^1(\R^n)$, then
\begin{equation}\label{vert2}
\sum_{j=0}^n||N_vf_j||_1\le c||f||_{H^1},
\end{equation}
where $f_0=f,$ $f_j=R_jf\,\,\,(j=1,...,n)$ (see \cite[Ch. VII.3.2]{St})).

Inequalities (\ref{max}) -- (\ref{vert2}) imply that for any $f\in H^1(\R^n)$ its Riesz transforms $R_jf$ ($j=1,...,n)$ belong to $H^1(\R^n)$ and
\begin{equation}\label{riesz}
||R_j f||_{H^1}\le c ||f||_{H^1}\quad(j=1,...,n)
\end{equation}
(see also \cite[pp.  288,  322]{GR}).

Denote by $HW_1^1(\R^n)$ the space of all functions $f\in H^1(\R^n)$ for which all weak partial derivatives
$D_jf$ exist and belong to $H^1(\R^n)$.

\begin{lem}\label{Supp} Let $f\in HW_1^1(\R^n)$ and let $u(x,t)=(P_t\ast f)(x),$ $t>0$. Set
\begin{equation}\label{sup1}
\widetilde{N}f(x)= \sup_{(y,t)\in \Gamma(x)}\left|\frac{\partial u}{\partial{t}}(y,t)\right|.
\end{equation}
Then
\begin{equation}\label{sup0}
\widetilde{N}f(x)\le \sum_{j=1}^n N(R_j(D_jf))(x)
\end{equation}
and
\begin{equation}\label{sup2}
||\widetilde{N}f||_1\le c\sum_{j=1}^n||D_jf||_{H^1}.
\end{equation}
\end{lem}
\begin{proof} Let $u_j(x,t)=P_t\ast(R_jf)(x)$ $(j=1,...,n)$. By the Fourier inversion,
$$
u_j(x,t)=-\int_{\R^n}\widehat f(\xi)\frac{i\xi_j}{|\xi|}e^{2\pi i\xi\cdot x}e^{-2\pi|\xi|t}d\xi.
$$
 Further,
$$
\frac{\partial u_j}{\partial{x_j}}(x,t)=-\int_{\R^n}2\pi i\xi_j\widehat f(\xi)\frac{i\xi_j}{|\xi|}e^{2\pi i\xi\cdot x}e^{-2\pi|\xi|t}d\xi.
$$
Indeed, differentiation under the integral sign is justified by the convergence of the integral
$$
\int_{\R^n}|\xi||\widehat f(\xi)|e^{-2\pi|\xi|t}d\xi, \quad t>0.
$$
Thus,
\begin{equation}\label{sup4}
\frac{\partial u_j}{\partial{x_j}}(x,t)=(P_t\ast (R_j(D_jf)))(x) \quad(j=1,...,n).
\end{equation}
By (\ref{conj}),
\begin{equation}\label{sup3}
\left|\frac{\partial u}{\partial{t}}(x,t)\right|\le \sum_{j=1}^n \left|\frac{\partial u_j}{\partial{x_j}}(x,t)\right|.
\end{equation}
Applying (\ref{sup3}) and (\ref{sup4}), we get (\ref{sup0}). By (\ref{max}) and (\ref{riesz}),
this implies (\ref{sup2}).
\end{proof}

As it was mentioned above, the following theorem holds.
\begin{teo}\label{Hardy1} Assume that  $f\in HW_1^1(\R^n)\,\,(n\in\N)$ and $1<q<n'.$  Then
$$
\sum_{k=1}^n\int_0^\infty h^{n/q'-1}||\Delta_k(h)f||_{q}\frac{dh}{h}\le c \sum_{k=1}^n||D_kf||_{H^1}.
$$
\end{teo}
For $n\ge 2$ this result follows from Theorem  \ref{refined}; for $n=1$ it was proved in \cite{Oswald} (see also \cite{K88}).

In this section we obtain a refinement of Theorem \ref{Hardy1} for $n\ge 2.$ For $1<q<\infty$ and $k=1,...,n$, denote by $V_{q,k}$
the mixed norm space
$L^{q,1}(\R^{n-1})[L^1(\R)]_k$
obtained by taking first the norm in $L^1(\R)$ with respect to the variable $x_k$, and then the norm in $L^{q,1}(\R^{n-1})$ with respect to $\widehat x_k.$
\begin{teo}\label{Hardy2} Assume that  $f\in HW_1^1(\R^n)\,\,(n\ge 2).$ Let $1<q<(n-1)'$ and $\a=1-(n-1)/q'.$  Then
\begin{equation}\label{hardy1}
\sum_{k=1}^n\int_0^\infty h^{-\a}||\Delta_k(h)f||_{V_{q,k}}\frac{dh}{h}\le c \sum_{k=1}^n||D_kf||_{H^1}.
\end{equation}
\end{teo}
\begin{proof} For $n\ge 3$ (\ref{hardy1}) follows from the stronger inequality (\ref{embed3}). We assume that $n=2.$ Set
$$
\f_h(x)=\int_\R|f(x,y+h)-f(x,y)|dy,\quad h>0.
$$
We consider the integral
\begin{equation}\label{hardy0}
J=\int_0^\infty h^{-1/q-1}\int_0^\infty s^{1/q-1}\f_h^*(s)dsdh.
\end{equation}
 We have
$$
\begin{aligned}
&J= \int_0^\infty h^{-1/q-1}\int_h^\infty s^{1/q-1}\f_h^*(s)dsdh\\
&+ \int_0^\infty h^{-1/q-1}\int_0^h s^{1/q-1}\f_h^*(s)dsdh\equiv J_1+J_2.
\end{aligned}
$$
As in Theorem \ref{refined}, we have the estimate
\begin{equation}\label{hardy2}
\f_h^*(s)\le hg^*(s), \quad \mbox{where}\quad g(x)=\int_\R|D_2f(x,y)|dy.
\end{equation}
Applying this estimate, we immediately get that
\begin{equation}\label{hardy3}
J_1\le c ||D_2f||_1.
\end{equation}

 Further, for the simplicity, we may assume that $J_2<\infty$ (otherwise we can apply the same  arguments as ones given  at the final part of the proof of Theorem \ref{refined} for estimation of $J_2$).
We first consider the difference $\f_h^*(s)-\f_h^*(2s).$
Denote
$$
\psi(x)=\int_\R N(D_1f)(x,y)dy,
$$
$$
 \psi_1(x)=\int_\R N(R_1(D_1f))(x,y)dy, \quad \psi_2(x)=\int_\R N(R_2(D_2f))(x,y)dy,
$$
and
$\Psi=\psi+\psi_1+\psi_2.$

Let $x\in \R$ and $s>0.$ There exists $\tau=\tau(x,s)\in (0,2s)$ such that
\begin{equation}\label{hardy4}
\f_h(x+2\tau)\le \f_h^*(2s) \quad\mbox{and}\quad \Psi(x+2\tau)\le \Psi^*(s).
\end{equation}
Indeed, let $A$ be the set of all $u\in (0, 4s)$ such that at least one of the inequalities
\begin{equation}\label{obratnye}
\f_h(x+u)> \f_h^*(2s) \quad\mbox{or}\quad \Psi(x+u)> \Psi^*(s)
\end{equation}
holds. Then $\mes_1 A\le 3s$ and therefore there exists $u\in (0, 4s)$ for which both the inequalities (\ref{obratnye}) fail.

Further, we have
$$
\f_h(x)-\f_h^*(2s)\le \f_h(x)-\f_h(x+2\tau)
$$
\begin{equation}\label{hardy6}
\le 2 \int_\R|f(x+2\tau,y)-f(x,y)|dy.
\end{equation}
For fixed $x,$  $y$, and $s$,  consider the cones
$$
\Gamma_1=\Gamma(x,y)\quad\mbox{and}\quad \Gamma_2=\Gamma(x+2\tau,y).
$$
The point $(x+\tau,y,\tau)$ belongs to  both of them. Let $u=P_t\ast f.$ Then
$$
\begin{aligned}
&|f(x+2\tau,y)-f(x,y)|\\
&\le |f(x,y)-u(x+\tau,y,\tau)|+|f(x+2\tau,y)-u(x+\tau,y,\tau)|\\
&\le \int_0^\tau\left(\left|\frac{\partial u}{\partial x}(x+t,y,t)\right|+\left|\frac{\partial u}{\partial t}(x+t,y,t)\right|\right)dt\\
&+\int_0^\tau\left(\left|\frac{\partial u}{\partial x}(x+\tau+s,y,\tau-s)\right|+\left|\frac{\partial u}{\partial t}(x+\tau+s,y,\tau-s)\right|\right)ds\\
&\le \tau\sup_{(x',y',t)\in \Gamma_1}\left(\left|\frac{\partial u}{\partial x}(x',y',t)\right|+\left|\frac{\partial u}{\partial t}(x',y',t)\right|\right)\\
&+\tau\sup_{(x',y',t)\in \Gamma_2}\left(\left|\frac{\partial u}{\partial x}(x',y',t)\right|+\left|\frac{\partial u}{\partial t}(x',y',t)\right|\right)\\
&\le\tau\left[N(D_1f)(x,y)+N(D_1f)(x+2\tau,y)+\widetilde{N}f(x,y)+\widetilde{N}f(x+2\tau,y)\right]
\end{aligned}
$$
(we have used the notation (\ref{sup1})). Applying  (\ref{sup0}), we have
$$
\begin{aligned}
&|f(x+2\tau,y)-f(x,y)|\le \tau\left(N(D_1f)(x,y) + N(D_1f)(x+2\tau,y)\right.\\
&\left.+N(R_1(D_1f))(x,y)+N(R_1(D_1f))(x+2\tau,y)\right.\\
&\left.+N(R_2(D_2f))(x,y)+N(R_2(D_2f))(x+2\tau,y)\right).
\end{aligned}
$$
By (\ref{hardy6}), this implies that
$$
\f_h(x)-\f_h^*(2s)
\le 2\tau(\Psi(x)+\Psi(x+2\tau)),
$$
where $\Psi=\psi+\psi_1+\psi_2.$ Taking into account (\ref{hardy4}), we obtain
$$
\f_h^*(s)-\f_h^*(2s)\le 8s\Psi^*(s).
$$
From here
$$
J_2'= \int_0^\infty h^{-1/q-1}\int_0^h s^{1/q-1}[\f_h^*(s)-\f_h^*(2s)]dsdh
$$
$$
\le 8\int_0^\infty h^{-1/q-1}\int_0^h s^{1/q}\Psi^*(s)dsdh=8q\int_0^\infty \Psi^*(s)ds=8q||\Psi||_1.
$$
Applying (\ref{max}) and (\ref{riesz}), we get
$$
\begin{aligned}
||\Psi||_1&=||N(D_1f)||_1+||N(R_1(D_1f))||_1+||N(R_2(D_2f))||_1\\
&\le c(||D_1f||_{H^1}+||D_2f||_{H^1}).
\end{aligned}
$$
Thus,
\begin{equation}\label{hardy7}
J_2'\le c'(||D_1f||_{H^1}+||D_2f||_{H^1}).
\end{equation}
Further, we consider
$$
J_2''=\int_0^\infty h^{-1/q-1}\int_0^h s^{1/q-1}\f_h^*(2s)dsdh.
$$
We have (see (\ref{hardy0}))
\begin{equation}\label{hardy8}
J_2''= 2^{-1/q}\int_0^\infty h^{-1/q-1}\int_0^{2h} s^{1/q-1}\f_h^*(s)dsdh
\le 2^{-1/q}J.
\end{equation}
Using estimates (\ref{hardy3}), (\ref{hardy7}), and (\ref{hardy8}),  we obtain
\begin{equation}\label{hardy100}
J\le 2^{-1/q}J + c(||D_1f||_{H^1}+||D_2f||_{H^1}).
\end{equation}
We assumed that $J_2<\infty$ and hence $J=J_1+J_2<\infty.$ Thus, (\ref{hardy100})  implies (\ref{hardy1}) for $n=2.$

\end{proof}

\vskip 6pt
\section{Estimates of Fourier transforms}

By Hardy's inequality, for any $f\in H^1(\R^n)$ $(n\in\N)$
\begin{equation}\label{H_ineq}
\int_{\R^n} \frac{|\widehat{f}(\xi)|}{|\xi|^n}\,d\xi\le c||f||_{H^1}.
\end{equation}

\vskip 4pt

 It was first discovered by Bourgain \cite{Bour1}  that for
$n\ge 2$ the Fourier transforms of the derivatives of
functions in the Sobolev space $W_1^1(\R^n)$ satisfy Hardy's
inequality. More exactly, Bourgain considered the periodic case.
His studies were continued by Pe\l czy\'nski and  Wojciechowski
\cite{PeWo}. The following theorem holds (Bourgain; Pe\l czy\'nski and  Wojciechowski).

\vskip 4pt
\begin{teo}\label{Pelcz} Let $f\in W_1^1(\R^n)$ $(n\ge 2).$ Then
\begin{equation}\label{pelcz}
\int_{\R^n} |\widehat{f}(\xi)||\xi|^{1-n}\,d\xi\le c
||\nabla f||_1.
\end{equation}
\end{teo}
\vskip 4pt

Equivalently,
\begin{equation}\label{pelcz1}
\sum _{k=1}^n\int_{\mathbb{R}^n}\frac{|(D_k
f)^\land(\xi)|}{|\xi|^n}d\xi \le c \sum _{k=1}^n\|D_k f\|_1.
\end{equation}
This is  Hardy type inequality. These results were extended in \cite{K1997}, \cite{K01}.

In contrast to (\ref{H_ineq}), inequalities  (\ref{pelcz}) and  (\ref{pelcz1}) fail to hold for $n=1.$

 Oberlin \cite{Ober}    proved the following refinement of Hardy's inequality (\ref{H_ineq}) valid for $n\ge 2$.

\begin{teo}\label{Oberlin} Let  $f\in H^1(\R^n)$ $(n\ge 2).$ Then
\begin{equation}\label{oberlin}
\sum_{k\in\Z} 2^{k(1-n)}\sup_{2^k\le r\le 2^{k+1}}\int_{S_r} |\widehat{f}(\xi)|\,d\s(\xi)\le c||f||_{H^1},
\end{equation}
where $S_r$ is the sphere of the radius $r$ centered at the origin in $\R^n$ and $d\s(\xi)$ is the canonical surface measure on $S_r.$
\end{teo}

Inequality (\ref{oberlin}) was used in \cite{Ober} to obtain the description of radial Fourier multipliers for $H^1(\R^n)$ $(n\ge 2).$
Observe that these results fail for $n=1.$

In this section we prove some estimates of Fourier transforms of functions in $W_1^1(\R^n)$ $(n\ge 3).$
In particular, these estimates provide Oberlin type inequalities for the Fourier transforms of the derivatives of
functions in $W_1^1(\R^n)$.

We shall use the notation (\ref{def**}).

\begin{teo}\label{Sup} Let $f\in W_1^1(\R^n)$ $(n\ge 3).$ Then
\begin{equation}\label{sup111}
\sum_{j=1}^n\int_0^\infty F_{t,j}^{**}(t^{n-1})\,dt\le c ||\nabla f||_1,
\end{equation}
where
\begin{equation}\label{sup11}
F_{t,j}(\widehat\xi_j)=\sup_{|\xi_j|\ge t}|\widehat f(\xi)|\quad (t>0).
\end{equation}
\end{teo}
\begin{proof} We estimate the first term of the sum in (\ref{sup111}). Set $\f_h(x)=\Delta_1(h)f(x).$ Then
$$
\widehat {\f_h}(\xi)=\widehat f(\xi)(e^{2\pi ih\xi_1}-1).
$$
Let $t>0$ and $\tau=1/t.$ Assume that $|\xi_1|\ge t.$ Then
$$
\frac{1}{\tau}\int_0^\tau |e^{2\pi ih\xi_1}-1|dh\ge \frac{1}{\tau}\int_0^\tau (1-\cos (2\pi \xi_1 h))dh
$$
$$
=1-\frac{\sin (2\pi\xi_1 \tau)}{2\pi\xi_1\tau}\ge 1-\frac{1}{2\pi|\xi_1|\tau}\ge 1-\frac{1}{2\pi}.
$$
It follows that
$$
\frac{2}{\tau}\int_0^\tau|\widehat {\f_h}(\xi)|dh\ge|\widehat f(\xi)|\quad\mbox{if}\quad |\xi_1|\ge t
$$
and
$$
\frac{2}{\tau}\sup_{|\xi_1|\ge t}\int_0^\tau|\widehat {\f_h}(\xi)|dh\ge F_{t,1}(\widehat\xi_1).
$$
By (\ref{supremum}), we have
$$
 F_{t,1}^{**}(t^{n-1})\le \frac{2t^{1-n}}{\tau}\sup_{\mes_{n-1}E=t^{n-1}}\sup_{|\xi_1|\ge t}\int_E\int_0^\tau|\widehat {\f_h}(\xi)|dhd\widehat\xi_1
$$
$$
\le \frac{2t^{1-n}}{\tau}\sup_{|\xi_1|\ge t}\int_0^\tau \sup_{\mes_{n-1}E=t^{n-1}}\int_E|\widehat {\f_h}(\xi)|d\widehat\xi_1dh.
$$
Let $1<q<(n-1)';$ then $q<2$. By H\"older's inequality, for any set $E\subset \R^{n-1}$ with $\mes_{n-1}E=t^{n-1}$ and any fixed $\xi_1\in \R$
$$
t^{1-n}\int_E|\widehat {\f_h}(\xi)|d\widehat\xi_1
\le t^{-(n-1)/q'}\left(\int_{\R^{n-1}}|\widehat {\f_h}(\xi)|^{q'}d\widehat\xi_1\right)^{1/q'}.
$$
Observe that for  fixed $h>0$ and  $\xi_1\in\R$, $\widehat {\f_h}(\xi)=(\widehat {\f_h})_{\xi_1}(\widehat \xi_1)$ is the Fourier transform of the function
$$
\widehat{x}_1\mapsto \int_\R \Delta_1(h)f(x)e^{-2\pi i\xi_1 x_1}\,dx_1.
$$
Applying the Hausdorff-Young inequality, we obtain
$$
\left(\int_{\R^{n-1}}|(\widehat {\f_h})_{\xi_1}(\widehat \xi_1)|^{q'}d\widehat\xi_1\right)^{1/q'}\le \left(\int_{\R^{n-1}}
\left|\int_\R \Delta_1(h)f(x)e^{-2\pi i\xi_1 x_1}\,dx_1\right|^q d\widehat x_1\right)^{1/q}
$$
$$
\le \left(\int_{\R^{n-1}}
\left(\int_\R |\Delta_1(h)f(x)|\,dx_1\right)^qd\widehat x_1\right)^{1/q}.
$$
Thus,
we have
$$
\frac{2t^{1-n}}{\tau}\sup_{\mes_{n-1}E=t^{n-1}}\sup_{|\xi_1|\ge t}\int_E\int_0^\tau|\widehat {\f_h}(\xi)|dhd\widehat\xi_1
$$
$$
\le 2t^{1-(n-1)/q'}\int_0^{1/t}||\Delta_1(h)f||_{L^q[L^1]}dh.
$$
It follows that
$$
\begin{aligned}
\int_0^\infty  F_{t,1}^{**}(t^{n-1})\,dt
&\le 2\int_0^\infty t^{1-(n-1)/q'}\int_0^{1/t}||\Delta_1(h)f||_{L^q[L^1]}dhdt\\
&= c \int_0^\infty h^{(n-1)/q'-1}||\Delta_1(h)f||_{L^q[L^1]}\frac{dh}{h}.
\end{aligned}
$$
Applying Theorem \ref{refined}, we obtain that
$$
\int_0^\infty  F_{t,1}^{**}(t^{n-1})\,dt\le c ||\nabla f||_1.
$$
\end{proof}

Similarly, we have the following theorem.
\begin{teo}\label{SupH} Let $f\in HW_1^1(\R^2).$  Then
\begin{equation*}
\int_0^\infty [F_{t,1}^{**}(t)+F_{t,2}^{**}(t)]dt\le c (||D_1f||_{H^1}+||D_2f||_{H^1}),
\end{equation*}
where
$$
F_{t,1}(\xi)=\sup_{|\eta|\ge t}|\widehat f(\xi,\eta)|,\quad F_{t,2}(\eta)=\sup_{|\xi|\ge t}|\widehat f(\xi,\eta)|.
$$
\end{teo}

Applying Theorem \ref{Sup}, we obtain the following Oberlin type estimate.
\begin{teo}\label{Obertype1} Let $f\in W_1^1(\R^n)$ $(n\ge 3).$ Then
\begin{equation}\label{obertype1}
 \sum_{k\in\Z} 2^{k(2-n)}\sup_{2^{k}\le r\le 2^{k+1}}\int_{S_r} |\widehat f(\xi)| d\s(\xi)\le c ||\nabla f||_1,
\end{equation}
where $S_r$ is the sphere of the radius $r$ centered at the origin in $\R^n$ and $d\s(\xi)$ is the canonical surface measure on $S_r.$
\end{teo}
\begin{proof} Let $B'_r$ be the ball  in $\R^{n-1}$ of the radius $r/\sqrt{n'}$ centered at the origin. Set
$$
S_{r,j}^+=\{\xi\in S_r:\xi_j\ge\frac{r}{\sqrt{n}}\}\quad\mbox{and}\quad S_{r,j}^-=\{\xi\in S_r:\xi_j\le -\frac{r}{\sqrt{n}}\}.
$$
Clearly,
\begin{equation}\label{obertype2}
S_{r,j}^+\cup S_{r,j}^-=\{\xi\in S_r: \widehat \xi_j\in B_r'\}\quad\mbox{and}\quad S_r=\bigcup_{j=1}^n(S_{r,j}^+\cup S_{r,j}^-).
\end{equation}
The surface $S^+_{r,j}$ is given by the equation
$$
\xi_j=\sqrt{r^2-|\widehat \xi_j|^2}, \quad \widehat \xi_j\in B'_r.
$$
Using notation (\ref{sup11}), we have
$$
\begin{aligned}
\int_{S_{r,j}^+}|\widehat f(\xi)| d\s(\xi)&=\int_{B_r'}\left|\widehat f\left(\sqrt{r^2-|\widehat \xi_j|^2}, \widehat \xi_j\right)\right|\frac{r}{\sqrt{r^2-|\widehat \xi_j|^2}}d\widehat\xi_j\\
&\le \sqrt{n}\int_{B_r'} F_{r/\sqrt{n},j}(\widehat\xi_j)d\widehat\xi_j.
\end{aligned}
$$
Further, $\mes_{n-1}B'_r=c_n r^{n-1}.$ If $2^k\le r\le 2^{k+1},$ then $\mes_{n-1}B'_r\asymp 2^{k(n-1)}.$ It easily follows that
$$
\begin{aligned}
&2^{k(1-n)}\sup_{2^{k}\le r\le 2^{k+1}}\int_{S^+_{r,j}} |\widehat f(\xi)| d\s(\xi)\\
&\le c 2^{k(1-n)} \int_0^{2^{k(n-1)}}F^*_{t_k,j}(u)du\le c'F^{**}_{t_k,j}(t_k^{n-1}),
\end{aligned}
$$
where $t_k=2^k/\sqrt{n}.$
Similar estimates hold for integrals over $S^-_{r,j}.$
Taking into account (\ref{obertype2}), we obtain
$$
\begin{aligned}
&\sum_{k\in\Z} 2^{k(2-n)}\sup_{2^{k}\le r\le 2^{k+1}}\int_{S_r} |\widehat f(\xi)| d\s(\xi)\\
&\le c\sum_{j=1}^n\sum_{k\in\Z} 2^k F^{**}_{t_k,j}(t_k^{n-1})\le c'\sum_{j=1}^n\int_0^\infty F_{t,j}^{**}(t^{n-1})dt.
\end{aligned}
$$
By  Theorem \ref{Sup}, this implies (\ref{obertype1}).

\end{proof}

We observe that (\ref{obertype1}) is equivalent to the inequality
$$
 \sum_{j=1}^n\sum_{k\in\Z} 2^{k(1-n)}\sup_{2^{k}\le r\le 2^{k+1}}\int_{S_r} |(D_jf)^\land(\xi)| d\s(\xi)\le c \sum_{j=1}^n||D_j f||_1
$$
which is a direct analogue of the Oberlin inequality (\ref{oberlin}).

Clearly, Theorem \ref{Sup} can be used to derive other Oberlin type estimates. For example, one can replace spheres by the surfaces of cubes.
For $k\in\Z$ and $1\le j\le n,$ denote
$$
Q_k^{(j)}=\{\widehat\xi_j: |\xi_m|\le 2^k, \,\, 1\le m\le n,\,\, m\not=j\}.
$$
Applying Theorem \ref{Sup}, we obtain the following
\begin{cor}\label{obertype3} Let $f\in W_1^1(\R^n)$ $(n\ge 3).$ Then
\begin{equation}\label{obertype33}
\sum_{j=1}^n \sum_{k\in\Z} 2^{k(2-n)}\sup_{2^{k}\le|\xi_j|\le 2^{k+1}}\int_{Q_k^{(j)}} |\widehat f(\xi)| d\widehat \xi_j\le c ||\nabla f||_1.
\end{equation}
\end{cor}

Let $Q_k=[-2^k, 2^k]^n$ and $P_k=Q_k\setminus Q_{k-1}\,\,(k\in\Z)$. We have
$$
\sum_{j=1}^n\sup_{2^{k-1}\le|\xi_j|\le 2^{k}}\int_{Q_k^{(j)}} |\widehat f(\xi)| d\widehat \xi_j\ge 2^{1-k}\int_{P_k} |\widehat f(\xi)| d\xi.
$$
Thus, (\ref{obertype33}) gives the strengthening of the inequality (\ref{pelcz}) (for $n\ge 3).$

\vskip 8pt

\noindent
{\bf Acknowledgments:} The author is grateful to  the referee for the careful revision which has greatly improved the final version of the work.

\end{document}